\documentclass[11pt]{article}
\usepackage{libertine}

\usepackage[letterpaper, top=25.4mm, bottom=25.4mm, left=25.4mm, right=25.4mm, includefoot]{geometry}

\usepackage{graphicx}
\usepackage{amsmath, amssymb, amsthm, dsfont}
\usepackage{paralist}
\usepackage{mathrsfs}
\usepackage{tikz}
\usepackage{aligned-overset}
\usepackage[ruled]{algorithm2e}
\usepackage{todonotes}
\usepackage{url}
\usepackage{comment}
\usepackage{nicefrac}
\usepackage{textpos}

\newcommand{\N}{\mathbb{N}}

\newcommand{\Oh}{\mathcal{O}}

\newcommand{\dist}{\mathsf{dist}}
\newcommand{\diam}{\mathsf{diam}}

\renewcommand{\subset}{\subseteq}
\renewcommand{\epsilon}{\varepsilon}

\theoremstyle{definition}
\newtheorem{definition}{Definition}

\theoremstyle{plain}
\newtheorem{theorem}{Theorem}
\newtheorem{corollary}{Corollary}
\newtheorem{lemma}{Lemma}

\newtheorem{observation}{Observation}

\usepackage{hyperref}
\usepackage[nameinlink]{cleveref}

\crefname{observation}{Observation}{Observations}
\crefname{lemma}{Lemma}{Lemmas}
\crefname{theorem}{Theorem}{Theorems}
\crefname{section}{Section}{Sections}
\crefname{definition}{Definition}{Definitions}
\crefname{corollary}{Corollary}{Corollaries}

\usepackage{xcolor}

\colorlet{myGreen}{green!50!black}
\colorlet{myLightgreen}{green}
\colorlet{myRed}{red!90!black}
\definecolor{myBlue}{rgb}{0.25, 0.0, 1.0}
\definecolor{myLightBlue}{rgb}{0.39, 0.58, 0.93}
\colorlet{myViolet}{myBlue!55!myRed}
\definecolor{myOrange}{rgb}{1.0, 0.66, 0.07}

\definecolor{CornflowerBlue}{rgb}{0.39, 0.58, 0.93}
\definecolor{DarkGoldenrod}{rgb}{0.72, 0.53, 0.04}
\definecolor{BritishRacingGreen}{rgb}{0.0, 0.26, 0.15}
\definecolor{DarkMagenta}{rgb}{0.55, 0.0, 0.55}
\definecolor{AO}{rgb}{0.0, 0.5, 0.0}
\definecolor{BostonUniversityRed}{rgb}{0.8, 0.0, 0.0}
\definecolor{myRed}{rgb}{0.8, 0.0, 0.0}
\definecolor{DarkMidnightBlue}{rgb}{0.0, 0.2, 0.4}
\definecolor{DarkTangerine}{rgb}{1.0, 0.66, 0.07}
\definecolor{AppleGreen}{rgb}{0.55, 0.71, 0.0}
\definecolor{BrightUbe}{rgb}{0.82, 0.62, 0.91}
\definecolor{Amethyst}{rgb}{0.6, 0.4, 0.8}
\definecolor{DarkGray}{rgb}{0.52, 0.52, 0.51}
\definecolor{Gray}{rgb}{0.66, 0.66, 0.66}
\definecolor{BananaYellow}{rgb}{1.0, 0.88, 0.21}
\definecolor{Amber}{rgb}{1.0, 0.75, 0.0}
\definecolor{LightGray}{rgb}{0.83, 0.83, 0.83}
\definecolor{PrincetonOrange}{rgb}{1.0, 0.56, 0.0}
\definecolor{DeepCarrotOrange}{rgb}{0.91, 0.41, 0.17}
\definecolor{CarrotOrange}{rgb}{0.93, 0.57, 0.13}
\definecolor{MidnightBlue}{rgb}{0.1, 0.1, 0.44}
\definecolor{Magenta}{rgb}{0.50, 0.0, 0.50}
\definecolor{BrightPink}{rgb}{1.0, 0.0, 0.5}
\definecolor{BrilliantRose}{rgb}{1.0, 0.33, 0.64}
\definecolor{ChromeYellow}{rgb}{1.0, 0.65, 0.0}
\definecolor{HotMagenta}{rgb}{1.0, 0.11, 0.81}

\definecolor{DarkTangerine}{rgb}{1.0, 0.66, 0.07}
\definecolor{darkyellow}{rgb}{.7, .6, 0.0}
\definecolor{CornflowerBlue}{rgb}{0.39, 0.58, 0.93}
\definecolor{DarkGoldenrod}{rgb}{0.72, 0.53, 0.04}
\definecolor{BritishRacingGreen}{rgb}{0.0, 0.26, 0.15}
\definecolor{AO}{rgb}{0.0, 0.5, 0.0}
\definecolor{MidnightBlack}{rgb}{0.1,0.1,.34}
\definecolor{MidnightBlue}{rgb}{0.1,0.1,0.43}
\definecolor{Black}{rgb}{0,0, 0}
\definecolor{Blue}{rgb}{0, 0 ,1}
\definecolor{Red}{rgb}{1, 0 ,0}
\definecolor{White}{rgb}{1, 1, 1}
\definecolor{DeepMagenta}{rgb}{0.8, 0.0, 0.8}
\definecolor{grey}{rgb}{.6, .6, .6}
\definecolor{darkgrey}{rgb}{.33, .33, .33}
\definecolor{Mygreen}{rgb}{.0, .7, .0}
\definecolor{Yellow}{rgb}{.55,.55,0}
\definecolor{Mustard}{rgb}{1.0, 0.86, 0.35}
\definecolor{applegreen}{rgb}{0.55, 0.71, 0.0}
\definecolor{darkturquoise}{rgb}{0.0, 0.81, 0.82}
\definecolor{celestialblue}{rgb}{0.29, 0.59, 0.82}
\definecolor{green_yellow}{rgb}{0.68, 1.0, 0.18}
\definecolor{crimsonglory}{rgb}{0.75, 0.0, 0.2}
\definecolor{darkmagenta}{rgb}{0.30, 0.0, 0.30}
\definecolor{magenta}{rgb}{0.50, 0.0, 0.50}
\definecolor{internationalorange}{rgb}{1.0, 0.31, 0.0}
\definecolor{darkorange}{rgb}{1.0, 0.55, 0.0}
\definecolor{ao}{rgb}{0.0, 0.5, 0.0}
\definecolor{awesome}{rgb}{1.0, 0.13, 0.32}
\definecolor{darkcyan}{rgb}{0.0, 0.50, 0.50}
\definecolor{violet}{rgb}{0.93, 0.51, 0.93}
\definecolor{brown}{rgb}{0.65, 0.16, 0.16}
\definecolor{orange}{rgb}{1.0, 0.65, 0.0}
\definecolor{DarkGreen}{rgb}{0,.5,0}
\definecolor{BostonUniversityRed}{rgb}{0.8, 0.0, 0.0}

\newcommand{\maxlayer}{\mathsf{maxlayer}}
\newcommand{\minlayer}{\mathsf{minlayer}}
\newcommand{\layer}{\mathsf{layer}}
\newcommand{\BFS}{\mathsf{BFS}}

\usepackage{todonotes}
\newcommand{\julia}[2][]{\todo[color=blue!20,#1]{\footnotesize {\textbf{Ju:} #2}}}

\newcommand{\aff}[1]{\textcolor{black!50}{#1}}

\hypersetup{
colorlinks=true,
linkcolor=darkorange!65!black,
citecolor=CornflowerBlue!65!black,
urlcolor=CornflowerBlue!65!black,
bookmarksopen=true,
bookmarksnumbered,
bookmarksopenlevel=2,
bookmarksdepth=3
}

\renewcommand{\leq}{\leqslant}
\renewcommand{\geq}{\geqslant}
\renewcommand{\le}{\leqslant}
\renewcommand{\ge}{\geqslant}

\title{\huge{A coarse block-cut tree theorem}\thanks{
\hspace{-0.66cm} \textbf{Funding information:}\\
JB (during employment in Warsaw), VB, MP: ERC project BOBR, funded from the European Research Council (ERC) under the European Union's Horizon 2020 research and innovation programme with grant agreement No. 948057.\\
JB (during employment in Oxford): ERC grant CCOO (grant no.~101165139)\\
JCz: Polish National Science Centre SONATA BIS-12, grant number 2022/46/E/ST6/00143.\\
EP: ERC grant BUKA (No. 101126229).
}
}

\author{
	Júlia Baligács\\{\small \aff{University of Oxford}}\\\href{mailto:jbaligacs@gmail.com}{\small jbaligacs@gmail.com}\bigskip
	\and
	Václav Blažej\\{\small \aff{University of Warsaw}}\\\href{mailto:v.blazej@uw.edu.pl}{\small v.blazej@uw.edu.pl}
	\and
	Jadwiga Czy\.zewska\\{\small \aff{University of Warsaw}}\\\href{mailto:j.czyzewska@mimuw.edu.pl}{\small j.czyzewska@mimuw.edu.pl}
	\and
	Michał Pilipczuk\\{\small \aff{University of Warsaw}}\\\href{mailto:michal.pilipczuk@mimuw.edu.pl}{\small michal.pilipczuk@mimuw.edu.pl}
	\and
	Evangelos Protopapas\\{\small \aff{University of Warsaw}}\\\href{mailto:eprotopapas@mimuw.edu.pl}{\small eprotopapas@mimuw.edu.pl}	
}

\date{}

\begin{document}
	\renewcommand\footnotemark{}
\maketitle
\begin{textblock}{20}(-1.6, 7)
	\includegraphics[width=35px]{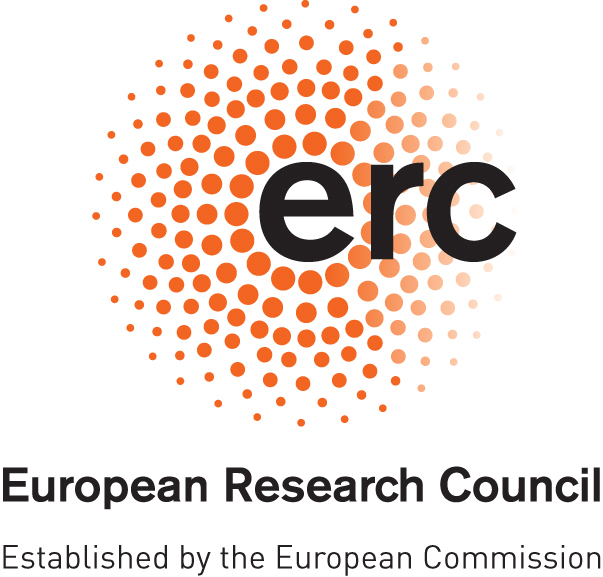}%
\end{textblock}
\begin{textblock}{20}(-1.6, 7.9)
	\includegraphics[width=35px]{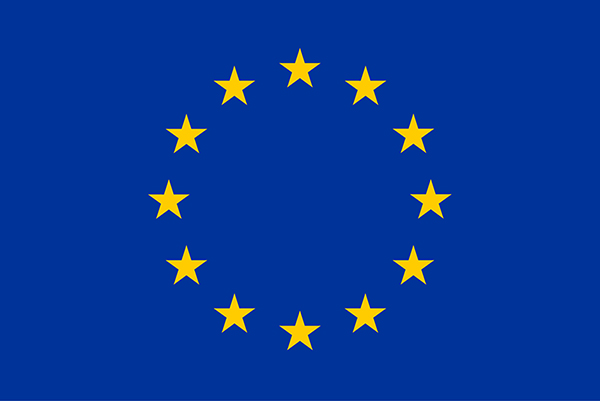}%
\end{textblock}

\begin{abstract}
We prove a coarse analogue of the classic fact that every graph can be decomposed along its cut-vertices into $2$-connected components. Precisely, we prove that for every graph $G$ and a positive integer~$d$, $G$ admits a tree decomposition whose adhesion sets have weak diameter at most $3d+2$ so that no two vertices $u,v$ lying in the same bag can be separated by a set of weak diameter at most $d$ whose distance from $u$ and $v$ is more than $d$. By the Coarse Menger's Theorem for two paths, this condition admits also a dual formulation, phrased in terms of the existence of two paths that are far from each other and connect the vicinity of $u$ with the vicinity of $v$.
\end{abstract}

\thispagestyle{empty}
\newpage

\section{Introduction}

\emph{Coarse graph theory} is a young and vibrant direction in structural graph theory that aims to understand the structure of graphs viewed as metric spaces, by endowing them with the shortest-paths metric. As outlined in the foundational work of Georgakopoulous and Papasoglu~\cite{GeorgakopoulosP2025Graph}, the hope is that many classic results of structural graph theory admit natural coarse analogues, obtained by replacing the conditions that objects intersect or are disjoint by requesting that they are close or far from each other. Recent investigations have revealed a mix of successes and failures. For instance, the Erd\H{o}s--P\'osa Theorem can be lifted to the coarse setting~\cite{AhnGHK25,fatCycles,DujmovicJMM24}, and so can also Gallai's Theorem~\cite{DistelGHLM26}. The coarse variant of Menger's Theorem fails in full generality~\cite{AS4:WeakMenger,AS0:StrongMenger}, but holds in various restricted settings~\cite{AlbrechtsenHJKW2024AMenger,CoarseMengerBPP26,GartlandKL23,GeorgakopoulosP2025Graph,HendreyNST24,Liu26,AS5:pw,AS6:disc}. There are also natural coarse analogues of pathwidth and treewidth~\cite{Hickingbotham25,AS1:CoarseTw,AS2:pwQI}. For these, the coarse analogue of the Excluded Tree Theorem holds~\cite{AS3:FatTree}, but the coarse analogue of the Excluded Grid Theorem fails~\cite{GridCounter}. The Alon-Seymour-Tarsi Separator Theorem can be lifted to the coarse setting, at least to a large degree~\cite{BonnetLPP26}. The coarse analogue of the Kuratowski-Wagner Theorem remains tantalizingly open~\cite{GeorgakopoulosP2025Graph}.

In this paper we add another positive example to this growing list. One of the most basic results of structural graph theory is the so-called \emph{block-cut tree}: every graph can be decomposed along its cut-vertices in a tree-like fashion into its $2$-connected components. We prove the following coarse analogue of this fact (see \cref{sec:prelims} for definitions):

\begin{theorem}\label{thm:block-cut-tree}
	For every positive integer $d$ and every connected graph $G$, there exists a tree-decomposition $(T, \beta)$ of $G$ such that
	\begin{itemize}
		\item each adhesion set of $(T, \beta)$ has weak diameter at most
		$3d + 2$; and
		\item 
		for any two vertices $u,v \in V(G)$ in a common bag of $(T, \beta)$, there is no set $S \subseteq V(G)$ with $\diam_G(S) \leq d$ and
		$\dist_G(S, \{ u, v\}) > d + 1$ that separates $u$ and $v$ in $G$.
	\end{itemize}
	Moreover, given $G$ and $d$, such a tree-decomposition can be computed in time $\Oh(n(n+m))$, where $n$ and $m$ denote the vertex and the edge count of $G$, respectively.
\end{theorem}

Since the coarse analogue of Menger's Theorem holds for two paths~\cite{AlbrechtsenHJKW2024AMenger,GeorgakopoulosP2025Graph}, the second condition in \cref{thm:block-cut-tree} can be reformulated to roughly the following: for any two vertices $u,v$ lying in the same bag and satisfying $\dist_G(u,v)\geq \Omega(d)$, there are two paths $P_1,P_2$ that connect the radius-$\Oh(d)$ balls around $u$ and~$v$, respectively, and are at distance $\Omega(d)$ from each other. See \cref{cor:block-cut-tree-dual} for a precise formulation.

The proof of \cref{thm:block-cut-tree} relies on a standard BFS-layering construction that has already been used in multiple previous contributions in coarse graph theory. The technique was perhaps best exemplified in the work of Berger and Seymour~\cite{BergerS2024Bounded}, who characterized graphs quasi-isometric to trees as those that admit a tree decomposition with bags of bounded weak diameter, and as those that exclude a certain obstruction called a \emph{loaded cycle}. The construction of Berger and Seymour is the following. Assuming without loss of generality that the considered graph $G$ is connected, we partition the vertices of $G$ into layers $L_0,L_1,L_2,\ldots$ according to the distances from some fixed vertex $r$. Each layer $L_i$ is subsequently partitioned into blocks, where vertices $u,v\in L_i$ are in the same block if they are connected in the graph induced by the union of layers $L_i,L_{i+1},L_{i+2},\ldots$. Importantly, the adjacency relation between all the blocks of all the layers induces a tree; this already provides a rough skeleton of a tree-decomposition. Berger and Seymour observe that in the absence of a heavily loaded cycle, every block in every layer has a bounded weak diameter, and thus the whole tree of blocks can be easily turned into a tree-decomposition of $G$ whose bags have bounded weak diameter. What we observe is that when the construction is applied to any graph, the blocks that do have bounded diameter are suitable ``coarse cut-vertices'', along which the graph can be decomposed into its ``coarse $2$-connected components''. 

We remark that a statement very similar to \cref{thm:block-cut-tree} has been independently proven by Albrechtsen and Georgakopoulos~\cite{AG26}.

\section{Preliminaries}\label{sec:prelims}

All graphs considered in this paper are undirected, finite, and simple. For $k\in \N$, we write $[k]\coloneqq \{1,\ldots,k\}$ and $[0,k]\coloneqq \{0,1,\ldots,k\}$.

\paragraph{Graphs.}

We use standard graph notation.
Given a graph $G$, we denote its set of vertices as $V(G)$ and its set of edges as $E(G)$.
An edge connecting vertices $u$ and $v$ is denoted as $uv$. 
For a subset of vertices $A$, we write $G[A]$ for the subgraph of $G$ \emph{induced} by $A$, which consists of all the vertices of $A$ and all the edges with both endpoints in $A$. 
The graph induced by the vertex subset $V(G)\setminus A$ is denoted by $G-A$.
We say that a vertex subset $S\subseteq V(G)$ \emph{separates} two vertices $u,v\notin S$ if $u$ and $v$ belong to distinct connected components of $G-S$.

\paragraph{Paths, distances, and neighborhoods.}

A $u$-$v$ \emph{path} in a graph $G$ is any path in $G$ with endpoints $u$ and~$v$.
For two subsets of vertices $S,\ T\subseteq V(G)$, an $S$-$T$ \emph{path} in $G$ is an $s$-$t$ path in $G$ for some $s\in S,\ t\in T$.
The \emph{distance metric} in $G$, defined between pairs of vertices $u, v$ of $G$ as the length of the shortest $u$-$v$ path in $G,$ will be denoted by  $\dist_G(u,v).$ If $A,B$ are vertex subsets or subgraphs of $G$, then $\dist_G(A,B)$ is the minimum of $\dist_G(a,b)$ for $a$ belonging to $A$ and $b$ belonging to $B$.

By $N(v)$ we denote the \emph{open neighborhood} of $v$ in $G$, defined as the set of vertices of $G$ which are adjacent to $v$. This notion is extended to any vertex subset $A \subseteq V(G)$ as follows: $N(A) \coloneqq \bigcup_{u\in A} N(u) \setminus A$. 
For $d\in \N$, the \emph{distance-$d$ neighborhoods} of $v$ and of $A$ are defined as
\[N^d[v] \coloneqq \{u\in V(G) \mid \dist_G(u,v)\leq d\}\qquad\textrm{and}\qquad N^d[A]\coloneqq \bigcup_{v\in A} N^d[v],\]
respectively.
Note that $N^1[v]=\{v\}\cup N(v)$. The \emph{weak diameter} of a vertex subset $S\subseteq V(G)$ is defined~as
\[\diam_G(S)\coloneqq \max_{u,v\in S}\ \dist_G(u,v).\]




For a connected graph $G$ and a subset of vertices $A\subseteq V(G)$, the \emph{$\mathsf{BFS}$-layering} of $G$ from $A$ is the partition of $V(G)$ into disjoint layers $L_0, L_1, L_2,\ldots,L_h$ so that $L_0=A$ and for every $i\geq 1$, $$L_i=N^i[A]\setminus N^{i-1}[A]=\{u\in V(G)\mid \dist_G(u,A)=i\}.$$
Note that for every $i\geq 1$ we have $L_i=N(L_{i-1})\setminus \bigcup_{j<i-1} L_j$, and in particular every vertex of $L_i$ has a neighbor in $L_{i-1}$. Moreover, whenever $uv$ is an edge of $G$ and $i,j$ are such that $u\in L_i$ and $v\in L_j$, we have $|i-j|\leq 1$.

\paragraph{Tree-decompositions.}

Given a graph $G$, we say that a pair $(T, \beta)$ is a \emph{tree decomposition} of $G$ if $T$ is a tree and $\beta$ is a function assigning each node $x$ of $T$ a subset of $\beta(x)\subseteq V(G)$, called the \emph{bag} of $x$, so that the following conditions hold:
\begin{itemize}
    \item for every vertex $u\in V(G)$, the set of nodes of $T$ whose bags contain $u$ induces a~non-empty connected subtree of $T$; and
    \item for every edge $uv\in E(G)$, there exists a node $x\in V(T)$ such that $u, v\in \beta(x)$.
\end{itemize}
Given two adjacent nodes $x$ and $y$ of $T$, the set $\beta(x)\cap \beta(y)$ is called the \emph{adhesion set} of the edge $xy \in E(T).$
The \emph{adhesion sets} of $(T,\beta)$ are the adhesion sets associated with the edges of $T$.
\section{A coarse analogue of the block-cut tree}


In this section we prove our main result, \cref{thm:block-cut-tree}.
We proceed to explain how to define the desired tree-decomposition and subsequently we verify that it satisfies the desired properties.

\paragraph{$\mathsf{BFS}$-layering trees.} Let us first recall the construction that originates from the work of Berger and Seymour~\cite{BergerS2024Bounded} on bounded-diameter tree-decompositions.

\medskip
Suppose $G$ is a connected graph and suppose that $\mathcal{L} \coloneqq \{ L_{0}, L_{1}, \ldots, L_{h} \}$, for some $h \in \mathbb{N}$, is the $\mathsf{BFS}$-layering of $G$ from an arbitrarily chosen (root) vertex $r \in V(G)$.

With $\mathcal{L}$ fixed in the context, we refer to the set $L_{i}$, $i \in [0, h]$, as the \emph{$i$-th layer} of $G$.
Given $i \in [0, h]$, we define relation $\sim_i$ on the set $L_i$ by declaring $u \sim_i v$ if there exists a $u$-$v$ path in $G[\bigcup_{j\ge i} L_j]$.
It is easy to see that this is an equivalence relation; hence $\sim\coloneqq \bigcup_i \sim_i$ is an equivalence relation on $V(G)$.
We call the equivalence classes of $\sim_i$ the \emph{blocks of layer $i$}, whereas the equivalence classes of $\sim$ are the \emph{blocks of $\mathcal{L}$}.

Let $T_{\mathcal{L}}$ be the graph whose vertices are the blocks of $\mathcal{L}$ and where two blocks $B,B'$ are adjacent if some vertex of $B$ is adjacent in $G$ to some vertex of $B'$.

We proceed with two basic observations about the graph $T_{\mathcal{L}}$.

The first one shows that all the edges of $G$ connecting a block of some layer to the previous layer in $\mathcal{L}$ lead to the same block; which shows that $T_{\mathcal{L}}$ is a tree.

\begin{observation}\label{obs:parent_block} Let $u',v'\in L_{i-1}$ and $u,v\in L_i$ with $uu', vv'\in E(G)$, for some $i \in [h]$.
Suppose $u$ and $v$ belong to the same block of layer $i$. Then $u'$ and $v'$ belong to the same block of layer $i-1$.
Consequently, the graph $T_{\mathcal{L}}$ is a tree.
\end{observation}
\begin{proof} Since $u$ and $v$ belong to the same block of layer $i$, there exists a path connecting them in $G[\bigcup_{j\ge i} L_j]$.
Appending the edges $uu'$ and $v,v'$ yields a $u'$-$v'$ walk in $G[\bigcup_{j\ge i-1} L_j]$, showing that $u'$ and $v'$ belong to the same block. Then this block is the unique neighbor in layer $L_{i-1}$ of the block of $u,v$, which implies that $T_{\mathcal{L}}$ is a tree.
\end{proof}

From now on we consider the tree $T_{\mathcal{L}}$ to be rooted at the 
the unique block $L_{0} = \{ r \}$ of layer $0$. This naturally imposes the parent--child and ancestor--descendant relations on the blocks.
The second observation shows that for any two vertices in different blocks that are in the ancestor--descendant relation in $T_{\mathcal{L}}$, there is a path between them in $G$ that avoids all the earlier layers.

\begin{observation}\label{obs:path_to_descendants} Let $B$ be a block of layer $i$ and let $B'$ be a descendant of $B$ in $T_{\mathcal{L}}$.
Then, for any two vertices $u\in B$ and $v\in B'$, there exists a $u$-$v$ path in $G[\bigcup_{j\ge i} L_j]$.
\end{observation}
\begin{proof} Consider a shortest $v$-$r$ path in $G$, say $P$.
Noting that $P$ must intersect layer $L_i$, let $w$ be the first vertex encountered on $P$ that belongs to $L_i$. Then the prefix of $P$ from $v$ to $w$ is entirely contained in $G[\bigcup_{j\ge i}L_j]$, and by \cref{obs:parent_block} we have that $w\in B$.
Since $u\in B$ as well, $u$ and $w$ belong to the same block of layer $i$ and hence there also exists a $w$-$u$ path in $G[\bigcup_{j\ge i}L_j]$.
Concatenating these paths gives a $u$-$v$ walk in $G[\bigcup_{j\ge i}L_j]$ and therefore also the desired $u$-$v$ path.
\end{proof}

\paragraph{Coarse block-cut tree-decomposition.}

Next, given a positive integer $d$, we say that a block of $\mathcal{L}$ is \emph{$d$-compact} if it has (weak) diameter at most $d$.
We also say that a block of $\mathcal{L}$ is \emph{$d$-wide} if it is not $d$-compact. We let $\mathcal{C}^d$ be the set of all $d$-compact blocks of $\mathcal{L}$

We are now in the position to define a graph $T^d_{\mathcal{L}}$ and a function $\beta^d_{\mathcal{L}} \colon V(T^d_{\mathcal{L}}) \to 2^{V(G)}$ as follows.
\begin{itemize}
\item $V(T^d_{\mathcal{L}})$ consists of the following nodes:
\begin{itemize}
\item A node $x_{B}$ for each $d$-compact block $B\in \mathcal{C}^d$. We set $\beta^d_{\mathcal{L}}(x_{B}) \coloneqq B$.
\item A node $x_{BB'}$ for each edge of $T_{\mathcal{L}}$ connecting two $d$-compact blocks $B,B'\in \mathcal{C}^d$. We set $\beta^d_{\mathcal{L}}(x_{BB'}) \coloneqq B \cup B'$.
\item A node $x_{C}$ for each connected component $C$ of $T_{\mathcal{L}}-\mathcal{C}^d$. We set $\beta^d_{\mathcal{L}}(x_{C})$ to be the union of $\bigcup_{B \in C} B$ and all the $d$-compact blocks of $\mathcal{L}$ adjacent to $C$ in $T_{\mathcal{L}}$.
\end{itemize}
\item $E(T^d_{\mathcal{L}})$ consists of the following edges:
\begin{itemize}
\item For each edge of $T_{\mathcal{L}}$ connecting two $d$-compact blocks $B,B'\in \mathcal{C}^d$, $E(T^d_{\mathcal{L}})$ contains the edges $u_{B}u_{BB'}$ and $u_{BB'}u_{B'}$.
\item For each connected component $C$ of $T_{\mathcal{L}}-\mathcal{C}^d$, $E(T^d_{\mathcal{L}})$ contains the edge $u_{C}u_{B}$, for each $d$-compact block of $\mathcal{L}$ adjacent to $C$ in $T_{\mathcal{L}}$.
\end{itemize}
\end{itemize}

In the following we observe that the pair $(T^d_{\mathcal{L}}, \beta^d_{\mathcal{L}})$ is a valid tree-decomposition of $G$ with adhesion sets of weak diameter at most $d$.

\begin{lemma}\label{obs:valid_decomposition} $(T^d_{\mathcal{L}}, \beta^d_{\mathcal{L}})$ is a tree-decomposition of $G$ with adhesion sets of weak diameter at most $d$.
\end{lemma}
\begin{proof} Recall that $T_{\mathcal{L}}$ is a tree, as shown in \cref{obs:parent_block}. It is then straightforward to see that by the construction, $T^d_{\mathcal{L}}$ is also a tree.

It remains to verify the two conditions from the definition of tree-decompositions.
First, observe that every vertex of $G$ belong to some bag, since it belongs to some block of $\mathcal{L}$ and every block of $\mathcal{L}$ is contained in some bag by construction.
Moreover, for every vertex $v \in V(G)$, the bags of containing $v$ form a subtree of $T^d_{\mathcal{L}}$: If $v$ belongs to a $d$-wide block, then it belongs only to the bag containing the vertices of the connected component of $T_{\mathcal{L}}-\mathcal{C}^d$ that contains $v$.
Otherwise, $v$ belongs to some $d$-compact block $B$ and therefore it belongs only to the bag $\beta^d_{\mathcal{L}}(x_{B})$ and the bags of the neighbors of $x_{B}$ in $T^d_{\mathcal{L}}$.

Next, for every edge $uv \in E(G)$, by definition of $\mathcal{L}$, the blocks of $\mathcal{L}$ containing $u$ and $v$ are equal or adjacent in $T_{\mathcal{L}}$. By construction, adjacent blocks are contained in a common bag of $(T^d_{\mathcal{L}}, \beta^d_{\mathcal{L}})$.
Hence, every edge of $G$ is contained in some bag. This verifies that $(T^d_{\mathcal{L}}, \beta^d_{\mathcal{L}})$ is tree-decomposition of $G$.

It remains to verify that each adhesion set of $(T^d_{\mathcal{L}}, \beta^d_{\mathcal{L}})$ has weak diameter at most $d$.
This follows from the construction, because the adhesion sets of $(T^d_{\mathcal{L}}, \beta^d_{\mathcal{L}})$ are precisely the $d$-compact blocks of $\mathcal{L}$.
\end{proof}

Given a $\mathsf{BFS}$-layering $\mathcal{L}$ of $G$ and a positive integer $d$, we call the tree-decomposition $(T^d_{\mathcal{L}}, \beta^d_{\mathcal{L}})$ defined as above, the \emph{$d$-coarse block-cut tree-decomposition} of $G$ with respect to $\mathcal{L}$. Let us note that this tree-decomposition can be computed efficiently.

\begin{observation}\label{obs:algo}
	Given $G$, $d$, and the $\mathsf{BFS}$-layering $\mathcal{L}$, the $d$-coarse block-cut tree-decomposition of $G$ can be computed in time $\Oh(n(n+m))$, where $n$ and $m$ denote the vertex and the edge count of $G$.
\end{observation}
\begin{proof}
	By applying breadth-first search from every vertex of $G$, in time $\Oh(n(n+m))$ we may compute the distance between every pair of vertices of $G$. Also, if for every vertex $u\in V(G)$ we apply breadth-first search from $u$ in the subgraph induced by the layer of $u$ and all the subsequent layers, then in total time $\Oh(n(n+m))$ we may compute the equivalence relation $\sim$, so also all the blocks of $\mathcal{L}$ and the tree $T_{\mathcal{L}}$. By combining the information above, we may in time $\Oh(n^2)$ recognize which blocks of $\mathcal{L}$ are $d$-compact and which are $d$-wide. Given this, it is straightforward to construct the tree-decomposition $(T^d_{\mathcal{L}}, \beta^d_{\mathcal{L}})$ in time $\Oh(n)$ using a depth-first search on $T_{\mathcal{L}}$.
\end{proof}

\paragraph{Towards the proof of \cref{thm:block-cut-tree}.}

To complete the proof of \cref{thm:block-cut-tree}, it remains to show that no bag of a $d$-coarse block-cut tree-decomposition contains a pair of vertices separable by a separator of small diameter (second condition).
To this end, we first identify two cases in which a vertex subset of small diameter does not separate a vertex from the root of the $\mathsf{BFS}$-layering.

For a given set of vertices $S \subseteq V(G)$, we define $\maxlayer_{\mathcal{L}}(S)$ to be the deepest layer containing a vertex of $S$, i.e., $$\maxlayer_{\mathcal{L}}(S) \coloneqq \max\{ i \in [h] \mid L_{i} \cap S \neq \emptyset \}.$$
Analogously, we define $\minlayer(S)$ as the minimum layer containing a vertex of $S$, i.e., $$\minlayer_{\mathcal{L}}(S) \coloneqq \min\{ i \in [h] \mid L_{i} \cap S \neq \emptyset \}.$$
For a block $B$ of $\mathcal{L}$, we have $\maxlayer_{\mathcal{L}}(B) = \minlayer_{\mathcal{L}}(B)$; we shall denote this common value by $\layer_{\mathcal{L}}(B)$.
For a vertex $v \in V(G)$, we define $\layer_{\mathcal{L}}(v) \coloneqq i \in [h]$, where $v \in L_{i}$.

We proceed by proving two important lemmas.

\begin{lemma}\label{lem:S_close}
Suppose $v \in V(G)$ and $B$ is the block of $\mathcal{L}$ containing $v$.
Suppose further that $S \subseteq V(G)$ is a vertex subset such that $\diam_{G}(S) \le d$,
$\dist_G(v, S) > d + 1$, and $\maxlayer_{\mathcal{L}}(S) \ge \layer_{\mathcal{L}}(v) - 1$. Then there exists a $v$-$r$ path in $G - S$.
\end{lemma}
\begin{proof}
Let $P$ be a shortest $v$-$r$ path in $G$.
If we enumerate the vertices of $P$ starting from $v$, the $i$-th vertex of $P$ belongs to layer $\layer_{\mathcal{L}}(v) - (i-1)$.
Since $\diam_{G}(S) \le d$, we have
\[
\minlayer_{\mathcal{L}}(S) \ge \maxlayer_{\mathcal{L}}(S) - d \ge \layer_{\mathcal{L}}(v)-(d+1).
\]
Consequently, $S$ can intersect $P$ only within the first $d+2$ vertices of $P$, that is, within those vertices whose distance from $v$ is at most $d+1$.
However, since
$\dist_G(S, v) > d + 1$ by assumption,
this is not the case and $P$ is a path in $G - S$.
\end{proof}

\begin{lemma}\label{lem:path_to_root} Let $v\in V(G)$ and $B$ be the block of $\mathcal{L}$ containing $v$.
Suppose that $S\subset V(G)$ is a vertex subset such that $\diam_G(S)\leq d$,
$\dist_G(v, S) > d + 1$, and $\maxlayer_\mathcal{L}(S) < \layer_{\mathcal{L}}(B)-1$.
Suppose further that $B'$ is the block of layer $\maxlayer_{\mathcal{L}}(S) + 1$ on the unique $B$-$\{ r\}$ path in $T_{\mathcal{L}}$ and $B'$ is 
$(3d + 2)$-wide.
Then there exists a $v$-$r$ path in $G - S$.
\end{lemma}
\begin{proof}
Since $B'$ is 
$(3d + 2)$-wide, it contains two vertices at distance more than
$3d + 2$ from each other.
Since $\diam(S) \leq d$, one of these two vertices is at distance more than
$\nicefrac{(3d + 2 - d)}{2} = d + 1$ from $S$.
Denote this vertex by $v^\star$.
Recall that $\layer_{\mathcal{L}}(v^*) = \layer_{\mathcal{L}}(B') = \maxlayer_{\mathcal{L}}(S) + 1$.
By \cref{obs:path_to_descendants}, there exists a $v$-$v^\star$ path in the graph $G[\bigcup_{j \geq \layer_{\mathcal{L}}(v^*)} L_j]$, which is a subgraph of $G - S$.
Moreover, \cref{lem:S_close} applied to $v^\star$ and $B'$ yields a $v^\star$-$r$ path in $G - S$.
The concatenation of these two paths gives a $v$-$r$ walk in $G - S$, and therefore also the desired path.
\end{proof}

Now we have all the ingredients to complete the proof of \cref{thm:block-cut-tree}.

\begin{proof}[Proof of \cref{thm:block-cut-tree}]
Letting $d' \coloneqq (3d + 2)$,
we show that the
$d'$-coarse block-cut tree-decomposition $\mathcal{T} \coloneqq (T^{d'}_{\mathcal{L}}, \beta^{d'}_{\mathcal{L}})$ with respect to the $\BFS$-layering of $G$ from an arbitrarily chosen (root) vertex $r \in V(G)$, satisfies the conditions of \cref{thm:block-cut-tree}.
By \cref{obs:valid_decomposition}, it is a valid tree-decomposition of $G$ with adhesion sets of weak diameter at most
$d'$, and by \cref{obs:algo} it can be computed in time $\Oh(n(n+m))$.
So it only remains to prove the second property: for every~$u,v \in V(G)$ in a common bag of $\mathcal{T}$, there is no set $S\subset V(G)$ separating $u$ and $v$ in $G$ with $\diam_G(S)\leq d$ and
$\dist_G(S, \{ u, v\}) > d + 1$.

Let $B_u$ be the block of $\mathcal{L}$ containing $u$, $B_v$ be the block of $\mathcal{L}$ containing $v$, and $\widehat{B}$ be the block of $\mathcal{L}$ that is the least common ancestor of $B_{u}$ and $B_{v}$ in $T_{\mathcal{L}}$.
If $\maxlayer_{\mathcal{L}}(S) < \layer_{\mathcal{L}}(\widehat{B})$, the theorem follows easily from \cref{obs:path_to_descendants}. Indeed, for any vertex $w \in \widehat{B}$, there is a $w$-$v$ path and a $w$-$u$ path in $G[\bigcup_{j\geq \layer_{\mathcal{L}}(\widehat{B})} L_j]$, which is a subgraph of $G - S$.

So assume from now on that $\maxlayer_{\mathcal{L}}(S) \geq \layer_{\mathcal{L}}(\widehat{B})$.
We show that then, there exists a $v$-$r$ path and a $u$-$r$ path in $G - S$, which completes the proof.
By symmetry, it suffices to argue this only for $u$.

If $\maxlayer_{\mathcal{L}}(S)\geq \layer_{\mathcal{L}}(u)-1$, the statement follows from \cref{lem:S_close}.
Otherwise, define $B'$ as in \cref{lem:path_to_root}, i.e., let $B'$ be the block of layer $\maxlayer_{\mathcal{L}}(S)+1$ on the unique $B_{u}$-$\{ r\}$ path of $T_{\mathcal{L}}$.
Observe that, with the current assumptions, we have $\layer_{\mathcal{L}}(\widehat{B}) < \layer_{\mathcal{L}}(B') < \layer_{\mathcal{L}}(B_u)$.
In particular,~$B'$ is an internal vertex on the unique $B_u$-$B_v$ path of $T_{\mathcal{L}}$.
By the definition of $\mathcal{T}$, no block of $\mathcal{L}$ that is an internal vertex on the unique $B_{u}$-$B_{v}$ path of $T_{\mathcal{L}}$ can be $d'$-compact.
Therefore, $B'$ is $d'$-wide.
By \cref{lem:path_to_root}, there exists a $u$-$r$ path in $G - S$.
By applying  symmetric arguments to $v$, we may obtain a $v$-$r$ path in $G - S$.
Concatenating these two paths completes the proof.
%
%
%
%
\end{proof}

\section{A dual formulation}

In this section we derive a dual formulation of \cref{thm:block-cut-tree}, replacing the second condition by an equivalent (up to constant factors) condition expressed in terms of the existence of scattered paths.

Towards this, we require the following result which provides a coarse analogue of Menger's Theorem for the special case where we ask for two scattered paths.

\begin{theorem}[\cite{AlbrechtsenHJKW2024AMenger,GeorgakopoulosP2025Graph}]\label{thm:Menger}
    Let $G$ be a graph and $S,T\subseteq V(G)$ be subsets of vertices of $G$. Suppose that for some $d\in \N$, there do not exist two $S$-$T$ paths $P_1,P_2$ such that $\dist_G(P_1,P_2)\geq d$.
    Then there exists a vertex $w$ of $G$ such that for every $S$-$T$ path $P$, we have $\dist_G(w,P)\leq 129d$.
\end{theorem}

The constant $129$ is from the work of Albrechtsen et al.~\cite{AlbrechtsenHJKW2024AMenger}; from now on we denote it by $c\coloneqq 129$.

\begin{definition}
    Let $G$ be a graph, $u,v\in V(G)$ be two vertices of $G$, and $d,\ell\in \N$. We say that
    \begin{itemize}
        \item $u$ and $v$ satisfy the \emph{$(d, \ell)$-cut condition} if there is a vertex subset $F\subseteq V(G)$ such that $\diam_G(F)\leq d$, $\dist_G(F,\{u,v\})>\ell$, and $F$ intersects every $u$-$v$ path; and
        \item $u$ and $v$ satisfy the \emph{$(d, \ell)$-flow condition} if at least one of the following assertions holds:
        \begin{itemize}
        \item $\dist_G(u, v) \leq 2\ell$; or
        \item in $G$ there exist paths $P_1,P_2$ with endpoints $u_1,v_1$ and $u_2,v_2$ respectively, such that we have $\dist_G(P_1,P_2)>d$ and $\dist_G(w_t, w)\leq \ell$ for all $w\in \{u,v\}$ and $t\in \{1,2\}$.
        \end{itemize}
    \end{itemize}
\end{definition}

The following lemma shows that thees two conditions are dual to each other in an approximate sense.

\begin{lemma}\label{cut_flow_duality}
    Let $G$ be a connected graph, $u,v\in V(G)$ be two vertices of $G$, and $d,\ell\in \N$. Then the following implications hold:
    \begin{itemize}
        \item If $u,v$ satisfy the $(d, \ell)$-cut condition, then they do not satisfy the $(d,\ell)$-flow condition.
        \item Assuming $\ell > 2cd$, if $u,v$ do not satisfy the $(d, \ell)$-flow condition, then they do satisfy the $(2cd, \ell - (2cd + 1))$-cut condition.
    \end{itemize}
\end{lemma}
\begin{proof}
    First we show that if $u, v$ satisfy the $(d, \ell)$-cut condition, then they do not satisfy the $(d, \ell)$-flow condition.
    By definition of the cut condition, there exists a set $F$ such that we have $\diam_G(F)\leq d$, $\dist_G(F,\{u,v\})>\ell$, and $F$ intersects every $u$-$v$ path.
    Note that since the distance from $F$ to both $u$ and $v$ is larger than $\ell$, it must be that $\dist_G(u, v) > 2\ell.$
    Now, suppose towards a contradiction that $u, v$ satisfy the $(d, \ell)$-flow condition.
    Then, there exist two paths $P_1$ and $P_2$ with endpoints $u_1,v_1$ and $u_2,v_2$ respectively, witnessing that the $(d, \ell)$-flow condition is satisfied.
    Let us consider a shortest $u$-$u_{1}$ path $Q_1$ and a shortest $v$-$v_1$ path $R_1$ respectively.
    Then $P'_1 \coloneqq Q_1\cup P_1\cup R_1$ is a $u$-$v$ walk.
    Analogously, we define $Q_2$, $R_2$, and~$P_2'$.
    By assumption, $F$ intersects $P_1'$ and $P_2'.$
    As the paths $Q_1,Q_2$ and $R_1,R_2$ are contained in $N^\ell[u]$ and $N^\ell[v]$ respectively, $F$ intersects none of these paths.
    Therefore $F$ must intersect both $P_1$ and~$P_2$.
    However, we have $\dist_G(P_1, P_2)>d$, which contradicts the assumption that $\diam_G(F) \leq d$.

    Now, we prove that if $u, v$ do not satisfy the $(d, \ell)$-flow condition  with $\ell > 2cd$, then they satisfy the $(2cd, \ell - (2cd + 1))$-cut condition.
    Define $S \coloneqq N^\ell[u]$ and $T \coloneqq N^\ell[v]$ respectively.
    By definition of the flow condition we have that $\dist_G(u, v) > 2\ell$ and moreover that there are no two $S$-$T$ paths which are at distance more than $d$ from each other.
    Hence, by \cref{thm:Menger}, there exists a vertex $w$ such that $F \coloneqq N^{cd}[w]$ intersects every $S$-$T$ path.
    Let $P$ be the shortest $N^\ell[u]$-$N^\ell[v]$ path in $G.$
    Since $\dist_G(u, v) > 2\ell$, it must be that $N^\ell[u]$ and $N^\ell[v]$ are disjoint, and in particular, it must be that every vertex of $P$ is at distance at least $\ell$ from both $u$ and $v.$
    Let $z$ be a vertex in $P\cap F$.
    As $\diam_G(F) \leq 2cd$, $\dist_G(u,z)\geq \ell$, and $\dist_G(v,z)\geq \ell$, no vertex in $N^{\ell-(2cd+1)}[u]$ and $N^{\ell-(2cd+1)}[v]$ belongs to $F$.
    Therefore, $\dist_G(F, \{ u, v\}) > \ell-(2cd+1).$
    Finally, observe that every $u$-$v$ path in $G$ contains an $S$-$T$ subpath.
    Therefore $F$ intersects every $u$-$v$ path in $G$ and, as a result $u$ and $v$, satisfy the $(2cd, \ell-(2cd+1))$-cut condition.
\end{proof}

Finally, using \cref{cut_flow_duality} we reformulate \cref{thm:block-cut-tree} by dualizing the coarse $2$-connectivity condition.

\begin{corollary}\label{cor:block-cut-tree-dual}
For every positive integer $d$ and every connected graph $G$, there exists a tree decomposition $(T, \beta)$ of $G$ such that
\begin{itemize}
\item each adhesion set of $(T, \beta)$ has weak diameter at most $3d + 2$; and
\item for any two vertices $u,v \in V(G)$ in a common bag of $(T, \beta)$,  one of the following~assertions~holds:
        \begin{itemize}
            \item $\dist_G(u, v) \leq 4(d + 1)$; or
            \item in $G$ there exist paths $P_1,P_2$, with endpoints $u_1,v_1$ and $u_2,v_2$ respectively, such that we have $\dist_G(P_1,P_2)> \lfloor \nicefrac{d}{258} \rfloor$ and $\dist_G(w_t, w)\leq 2(d + 1)$ for all $w\in \{u,v\}$ and $t\in \{1,2\}$.
        \end{itemize}
\end{itemize}
\end{corollary}
\begin{proof} The second condition featured in \cref{thm:block-cut-tree} is equivalent to the following statement: For any two vertices $u,v \in V(G)$ in a common bag of $(T, \beta)$, $u, v$ do not satisfy the $(d, d+1)$-cut condition.
Then, by \cref{cut_flow_duality} we conclude that any two vertices in a common bag of $(T, \beta)$ satisfy the $(\lfloor \nicefrac{d}{2c} \rfloor, 2(d + 1))$-flow condition.
\end{proof}


\bibliographystyle{abbrv}
\bibliography{ref}

\end{document}